\documentclass{patmorin}
\usepackage{pat}
\usepackage{paralist}
\usepackage{dsfont}  %
\usepackage[utf8x]{inputenc}
\usepackage[T1]{fontenc}
\usepackage{skull}
\usepackage{paralist}
\usepackage{graphicx}
\usepackage{amssymb}

\allowdisplaybreaks

\usepackage[normalem]{ulem}
\usepackage{cancel}

\usepackage{todonotes}

\usepackage{etoolbox}

\usepackage[longnamesfirst,numbers,sort&compress]{natbib}

\usepackage[mathlines]{lineno}
\newcommand*\patchAmsMathEnvironmentForLineno[1]{%
 \expandafter\let\csname old#1\expandafter\endcsname\csname #1\endcsname
 \expandafter\let\csname oldend#1\expandafter\endcsname\csname end#1\endcsname
 \renewenvironment{#1}%
    {\linenomath\csname old#1\endcsname}%
    {\csname oldend#1\endcsname\endlinenomath}}%
\newcommand*\patchBothAmsMathEnvironmentsForLineno[1]{%
 \patchAmsMathEnvironmentForLineno{#1}%
 \patchAmsMathEnvironmentForLineno{#1*}}%
\AtBeginDocument{%
\patchBothAmsMathEnvironmentsForLineno{equation}%
\patchBothAmsMathEnvironmentsForLineno{align}%
\patchBothAmsMathEnvironmentsForLineno{flalign}%
\patchBothAmsMathEnvironmentsForLineno{alignat}%
\patchBothAmsMathEnvironmentsForLineno{gather}%
\patchBothAmsMathEnvironmentsForLineno{multline}%
}

\DeclareMathOperator{\tw}{tw}
\DeclareMathOperator{\dist}{dist}
\DeclareMathOperator{\diam}{diam}

\newcommand\subsetcong{\mathrel{\text{%
    \setbox0\hbox{$\subseteq$}%
    \rlap{\hbox to \wd0{\hss\hss\hss\raisebox{1.5\height}{$\sim$}\hss}}\box0
}}}
\renewcommand{\subsetcong}{\subseteq}

\renewcommand{\le}{\leqslant}

\renewcommand{\ge}{\geqslant}
\renewcommand{\geq}{\geqslant}

\title{\MakeUppercase{Bounded-Degree Planar Graphs Do Not Have Bounded-Degree Product Structure}}

\author{%
  Vida Dujmović\thanks{Department of Computer Science and Electrical Engineering, University of Ottawa, Ottawa, Canada.}\, \thanks{Research partly supported by NSERC.} \quad
  Gwenaël Joret\thanks{D\'epartement d'Informatique, Universit\'e Libre de Bruxelles, Brussels, Belgium.
  Supported by a CDR grant and a PDR grant from the Belgian National Fund for Scientific Research (FNRS).} \quad
  Piotr Micek\thanks{Theoretical Computer Science Department, Faculty of Mathematics and Computer Science, Jagiellonian University, Krak\'ow, Poland.}\quad
  Pat Morin\thanks{School of Computer Science, Carleton University, Ottawa, Canada}\, \footnotemark[2] \quad
  David R. Wood\thanks{School of Mathematics, Monash University, Melbourne, Australia.
  Research supported by the Australian Research Council.}
}

\date{}

\begin{document}
\maketitle

\begin{abstract}
   Product structure theorems are a collection of recent results that have been used to resolve a number of longstanding open problems on planar graphs and related graph classes.  One particularly useful version states that every planar graph $G$ is contained in the strong product of a $3$-tree $H$, a path $P$, and a $3$-cycle $K_3$; written as $G\subsetcong H\boxtimes P\boxtimes K_3$.  A number of researchers have asked if this theorem can be strengthened so that the maximum degree in $H$ can be bounded by a function of the maximum degree in $G$.  We show that no such strengthening is possible.  Specifically, we describe an infinite family $\mathcal{G}$ of planar graphs of maximum degree $5$ such that, if an $n$-vertex member $G$ of $\mathcal{G}$ is isomorphic to a subgraph of $H\boxtimes P\boxtimes K_c$ where $P$ is a path and $H$ is a graph of maximum degree $\Delta$ and treewidth $t$, then $t\Delta c \ge 2^{\Omega(\sqrt{\log\log n})}$.
\end{abstract}

\section{Introduction}

Recently, product structure theorems have been a key tool in resolving a number of longstanding open problems on planar graphs.  Roughly, a \defin{product structure theorem} for a graph family $\mathcal{G}$ states that every graph in $\mathcal{G}$ is isomorphic to a subgraph of the product of two or more ``simple'' graphs.  As an example, there are a number of graph classes $\mathcal{G}$ for which there exists integers $t$ and $c$ such that, for each $G\in\mathcal{G}$ there is a graph $H$ of treewidth\footnote{A \defin{tree decomposition} of a graph $H$ is a collection $\mathcal{T}:=(B_x:x\in V(T))$ of subsets of $V(H)$ indexed by the nodes of some tree $T$ such that
\begin{inparaenum}[(i)]
  \item for each $v\in V(H)$, the induced subgraph $T[x\in V(T):v\in B_x]$ is connected; and
  \item for each edge $vw\in E(H)$, there exists some $x\in V(T)$ with $\{v,w\}\subsetcong B_x$.
\end{inparaenum}
The \defin{width} of such a tree decomposition is $\max\{|B_x|:x\in V(T)\}-1$. The \defin{treewidth} of $H$ is the minimum width of any tree decomposition of $H$.} $t$ and a path $P$ such that $G$ is isomorphic to a subgraph of the strong product\footnote{The \defin{strong product} $G_1\boxtimes G_2$ of two graphs $G_1$ and $G_2$ is the graph with vertex-set $V(G_1\boxtimes G_2):=V(G_1)\times V(G_2)$ and that includes the edge with endpoints $(v,x)$ and $(w,y)$ if and only if
\begin{inparaenum}[(i)]
  \item $vw\in E(G_1)$ and $x=y$;
  \item $v=w$ and $xy\in E(G_2)$; or
  \item $vw\in E(G_1)$ and $xy\in E(G_2)$.
\end{inparaenum}
} of $H$, $P$, and a clique $K$ of order $c$.
This is typically written as $G\subsetcong H\boxtimes P\boxtimes K_c$, where the notation $G_1\subsetcong G_2$ is used to mean that $G_1$ is isomorphic to some subgraph of $G_2$.  See references \cite{dujmovic.joret.ea:planar,dujmovic.morin.ea:structure,krauthgamer.lee:intrinsic,ueckerdt.wood.ea:improved,bose.morin.ea:optimal,campbell.clinch.ea:product,illingworth.scott.ea:alon,distel.hickingbotham.ea:improved,hickingbotham.jungeblut.ea:product,hickingbotham.wood:shallow,wood:product} for examples.

In some applications of product structure theorems it is helpful if, in addition to having treewidth $t$, the graph $H$ has additional properties, possibly inherited from $G$.  For example, one very useful version of the planar graph product structure theorem states that for every planar graph $G$ there exists a \emph{planar} graph $H$ of treewidth $3$ and a path $P$ such that $G\subsetcong H\boxtimes P\boxtimes K_3$ \cite[Theorem~36(b)]{dujmovic.joret.ea:planar}.  The planarity of $H$ in this result has been leveraged to obtain better constants and even asymptotic improvements for graph colouring and layout problems, including queue number \cite{dujmovic.joret.ea:planar}, $p$-centered colouring \cite{debski.felsner.ea:improved}, and $\ell$-vertex ranking \cite{bose.dujmovic.ea:asymptotically}.

In this vein, the authors have been repeatedly asked if $H$ can have bounded degree whenever $G$ does; that is:
\begin{quote}
  For each $\Delta\in\N$, let $\mathcal{G}_\Delta$ be the family of planar graphs of maximum degree $\Delta$.  Do there exist functions $t:\N\to\N$, $d:\N\to\N$, and $c:\N\to\N$ such that, for each $\Delta\in\N$ and each $G\in\mathcal{G}_\Delta$ there exists a graph $H$ of treewidth at most $t(\Delta)$ and maximum degree $d(\Delta)$ and a path $P$ such that $G\subsetcong H\boxtimes P\boxtimes K_{c(\Delta)}$?
\end{quote}
In the current paper we show that the answer to this question is no, even when $\Delta=5$.

\begin{thm}\label{main_thm}
  For infinitely many integers $n\ge 1$, there exists an $n$-vertex planar graph $G$ of maximum degree $5$ such that, for every graph $H$ of treewidth $t$ and maximum degree $\Delta$, every path $P$, and every integer $c$, if $G\subsetcong H\boxtimes P\boxtimes K_c$ then $t\Delta c \ge 2^{\Omega(\sqrt{\log\log n})}$.
\end{thm}

The graph family $\mathcal{G}:=\{G_h:h\in\N\}$ that establishes \cref{main_thm} consists of complete binary trees of height $h$ augmented with edges to form, for each $i\in\{1,\ldots,h\}$, a path $D_i$ that contains all vertices of depth $i$.  See \cref{G_5}.

\begin{figure}
  \begin{center}
    \includegraphics{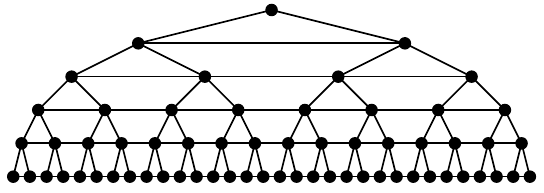}
  \end{center}
  \caption{The graph $G_5$ from the graph family $\{G_h:h\in\N\}$ that establishes \cref{main_thm}.}
  \label{G_5}
\end{figure}

\section{Proof of \cref{main_thm}}

Throughout this paper, all graphs $G$ are simple and undirected with vertex-set $V(G)$ and edge-set $E(G)$.  For a set $S$, $G[S]$ denotes the subgraph of $G$ induced by $S\cap V(G)$ and $G-S:=G[V(G)\setminus S]$.  For every $v\in V(G)$, let $N_G(v):=\{w:vw\in E(G)\}$ and for every $S\subseteq V(G)$, let $N_G(S):=\bigcup_{v\in S} N_G(v)\setminus S$.  We write $G_1\cong G_2$ if $G_1$ and $G_2$ are isomorphic and $G_1\subsetcong G_2$ if $G_1$ is isomorphic to some subgraph of $G_2$.  Inside of asymptotic notation, $\log n:=\max\{1,\log_2 n\}$.

\subsection{Partitions}

Let $G$ and $H$ be graphs.  An \defin{$H$-partition} $\mathcal{H}:=\{B_x:x\in V(H)\}$ of $G$ is a partition of $V(G)$ whose parts are indexed by the vertices of $H$ with the property that, if $vw$ is an edge of $G$ with $v\in B_x$ and $w\in B_y$ then $x=y$ or $xy\in E(H)$.  The \defin{width} of $\mathcal{H}$ is the size of its largest part; that is, $\max\{|B_x|:x\in V(H)\}$.
If $H$ is in a class $\mathcal{G}$ of graphs then we may call $\mathcal{H}$ a $\mathcal{G}$-partition of $G$.  Specifically, if $H$ is a tree, then $\mathcal{H}$ is a \defin{tree-partition} of $G$ and if $H$ is a path, then $\mathcal{H}$ is a \defin{path-partition} of $G$.  A path-partition $\mathcal{P}:=\{P_x:x\in V(P)\}$ of $G$ is also referred to as a \defin{layering} of $G$ and the parts of $\mathcal{P}$ are referred to as \defin{layers}.  A set of layers $\{P_{x_1},\ldots,P_{x_q}\}\subseteq\mathcal{P}$ is \defin{consecutive} if $P[\{x_1,\ldots,x_q\}]$ is connected.

As in previous works, we make use of the following relationship between $H$-partitions and strong products, which follows immediately from the preceding definitions.

\begin{obs}\label{partitions_vs_products}
  For every integer $c\ge 1$, and all graphs $G$, $H$, and $J$,  $G\subsetcong H\boxtimes J\boxtimes K_c$ if and only if $G$ has an $H$-partition $\mathcal{H}:=\{B_x:x\in V(H)\}$ and a $J$-partition $\mathcal{J}:=\{C_y:y\in V(J)\}$ such that $|B_x\cap C_y|\le c$, for each $(x,y)\in V(H)\times V(J)$.
\end{obs}

The following important result of \citet{ding.oporowski:some} (also see \cite{wood:on,distel.wood:tree_partitions}) allows us to focus on the case where the first factor in our product is a tree.

\begin{thm}[\citet{ding.oporowski:some}]\label{dingy}
  If $H$ is a graph with maximum degree $\Delta$ and treewidth $t$,
  then $H$ has a tree-partition of width at most $24\Delta(t+1)$.
\end{thm}

\begin{cor}\label{ding_translation}
  If $G\subsetcong H\boxtimes P\boxtimes K_c$ where $H$ has treewidth $t$ and maximum degree $\Delta$ then there exists a tree $T$ such that $G\subsetcong T\boxtimes P\boxtimes K_{24c\Delta (t+1)}$.
\end{cor}

\begin{proof}
  By \cref{dingy}, $H$ has a tree-partition $\mathcal{T}:=\{B_x:x\in V(T)\}$ of width at most $24\Delta (t+1)$. By \cref{partitions_vs_products}, $H \subsetcong T\boxtimes K_{24\Delta (t+1)}$.  Therefore, $G\subsetcong T\boxtimes K_{24\Delta (t+1)}\boxtimes P\boxtimes K_c \cong T\boxtimes P\boxtimes K_{24c\Delta (t+1)}$.
\end{proof}

The \defin{length} of a path is the number of edges in it. Given two vertices $v,w\in V(G)$, $\dist_G(v,w)$ denotes the minimum length of a path in $G$ that contains $v$ and $w$, or $\dist_G(v,w)$ is infinite if $v$ and $w$ are in different connected components of $G$. For any $R\subseteq V(G)$, the \defin{diameter} of $R$ in $G$ is $\diam_G(R):=\max\{\dist_G(v,w):v,w\in R\}$.

\begin{obs}\label{diameter_spread}
  Let $G$ be a graph, let $R\subseteq V(G)$, and let $\mathcal{L}$ be a layering of $G$.  Then there exists a layer $L\in\mathcal{L}$ such that $|R\cap L|\ge |R|/(\diam_G(R)+1)$.
\end{obs}

\begin{proof}
  By the definition of layering, the vertices in $R$ are contained in a set of at most $\diam_G(R)+1$ consecutive layers of $\mathcal{L}$. The result then follows from the Pigeonhole Principle.
\end{proof}

We also make use of the following basic fact about tree-partitions:

\begin{obs}\label{tree_thingy}
  Let $G$ be a graph, let $\mathcal{T}:=\{B_x:x\in V(T)\}$ be a tree-partition of $G$, let $x\in V(T)$, and let $v,w\in N_G(B_x)$ be in the same component of $G-B_x$.  Then $T$ contains an edge $xy$ with $v,w\in B_y$.
\end{obs}

\begin{proof}
  Suppose that $v\in B_y$ and $w\in B_z$ for some $y,z\in V(T)$.  Since $v,w\in N_G(B_x)$, $T$ contains the edges $xy$ and $xz$.  All that remains is to show that $y=z$. For the purpose of contradiction, assume $y\neq z$.  Since $v$ and $w$ are in the same component of $G-B_x$, $G$ contains a path from $v$ to $w$ that avoids all vertices in $B_x$, which implies that $T$ contains a path $P_{yz}$ from $y$ to $z$ that does not include $x$.  This is a contradiction since then $P_{yz}$ and the edges $xy$ and $yz$ form a cycle in $T$, but $T$ is a tree.
\end{proof}

\subsection{Percolation in Binary Trees}

The \defin{depth} of a vertex $v$ in a rooted tree $T$ is the length of the path $A_T(v)$ from $v$ to the root of $T$.  Each vertex $a\in V(A_T(v))$ is an \defin{ancestor} of $v$, and $v$ is a \defin{descendant} of each vertex in $V(A_T(v))$.  We say that a set $B\subseteq V(T)$ is \defin{unrelated} if no vertex of $B$ is an ancestor of any other vertex in $B$.

For each $h\in\N$, let $T_h$ denote the complete binary tree of height $h$; that is, the rooted ordered tree with $2^h$ leaves, each having depth $h$ and in which each non-leaf vertex has exactly two children, one \defin{left child} and one \defin{right child}.  Note that the ordering of $T_h$ induces an ordering on every unrelated set $B\subseteq V(T_h)$, which we refer to as the \defin{left-to-right ordering}.  Specifically, $v\in B$ appears before $w\in B$ in the left-to-right ordering of $B$ if and only if there exists a common ancestor $a$ of both $v$ and $w$ such that the path from $a$ to $v$ contains the left child of $a$ and the path from $a$ to $w$ contains the right child of $a$.

We use the following two percolation-type results for $T_h$.

\begin{lem}\label{one_path}
  Let $h\ge 1$, let $r$ be the root of $T_h$, and
  let $S\subseteq V(T_h)$ with $1\le |S|< 2^h$. Then there exists a vertex $v$ of $T_h$ such that
  \begin{compactenum}[(i)]
    \item the depth of $v$ is at most $\log_2|S|+1$;
    \item $v\neq r$ and the parent of $v$ is in $S\cup\{r\}$; and
    \item $T_h-S$ contains a path from $v$ to a leaf of $T_h$.
  \end{compactenum}
\end{lem}

\begin{proof}
  The proof is by induction on $h$.  When $h=1$, $|S|\le 1$. In particular, at least one child $v$ of $r$ is not in $S$.  The depth of $v$ is $1\le \log_2|S|+1$, so $v$ satisfies (i).  The parent of $v$ is $r\in S\cup\{r\}$, so $v$ satisfies (ii).  $T_1-S$ contains a length-$0$ path from $v$ to itself (a leaf of $T_1$), so $v$ satisfies (iii).

  For $h\ge 2$, let $\ell$ be the maximum integer such that $S\cup\{r\}$ contains all $2^\ell$ vertices of depth $\ell$.  Observe that $2^\ell \le |S|$, so $\ell \le \log_2 |S| < h$.  Let $L$ be the set of $2^{\ell+1}$ depth-$(\ell+1)$ vertices in $T_h$.  By the Pigeonhole Principle some vertex $r'\in L$ is the root of a complete binary tree $T'$ with root $r'$ of height $h-\ell-1$ with $|S\cap V(T')| \le |S|/2^{\ell+1} < 2^{h-\ell-1}$.

  If $V(T')\cap S=\emptyset$ then choosing $v:=r'$ satisfies the requirements of the lemma.  Otherwise, by applying induction on $T'$ and $S':=S\cap V(T')$ we obtain a vertex $v'$ of depth at most $\ell+1+\log_2(|S'|)+1 \le \log_2|S|+1$ whose parent is in $S\cup\{r'\}$, and such that $T_h-S$ contains a path from $v'$ to a leaf of $T_h$.  Thus $v'$ satisfies requirements (i) and (iii).  If the parent of $v'$ is in $S$ then $v'$ also satisfies requirement (ii) and the lemma is proven, with $v:=v'$.  Otherwise, the parent of $v'$ is $r'$, in which case $r'$ satisfies requirements (i)--(iii) and we are done, with $v:=r'$.
\end{proof}

\begin{lem}\label{two_paths}
  Let $h\ge 1$, let $r$ be the root of $T_h$, and
  and let $S\subseteq V(T_h)$ with $1\le |S|< 2^{h-1}$. Then there exist two unrelated vertices $v_1$ and $v_2$ of $T_h$ such that, for each $i\in\{1,2\}$:
  \begin{compactenum}[(i)]
    \item the depth of $v_i$ is at most $\log_2|S|+2$;
    \item $v_i\neq r$ and the parent of $v_i$ is in $S\cup\{r\}$; and
    \item $T_h-S$ contains a path from $v_i$ to a leaf of $T_h$.
  \end{compactenum}
\end{lem}

\begin{proof}
  Let $T_1$ and $T_2$ be the two maximal subtrees of $T_h$ rooted at the children $r_1$ and $r_2$, respectively of $r$. (Each of $T_1$ and $T_2$ is a complete binary tree of height $h-1$.)  For each $i\in\{1,2\}$, let $S_i:=S\cap V(T_i)$.  If $S_i=\emptyset$ then we choose $v_i=r_i$ and this satisfies requirements (i)--(iii).  If $S_i\neq\emptyset$ then, since $|S_i|\le |S|< 2^{h-1}$, we can apply \cref{one_path} to $T_i$ and $S_i$ to obtain a vertex $v_i'\in V(T_i)$ of depth at most $1+\log_2|S_i|+1 \le \log_2 |S| + 2$ and such that $T_h-S$ contains a path from $v_i'$ to a leaf of $T_h$.  Therefore, $v_i'$ satisfies (i) and (iii).  Furthermore, the parent of $v_i'$ is in $S\cup\{r_i\}$.  If the parent of $v_i'$ is in $S$, then $v_i'$ also satisfies (ii), so we set $v_i:=v_i'$.  If the parent of $v_i'$ is not in $S$, then the parent of $v_i'$ is $r_i\not\in S$ and $r_i$ satisfies (i)--(iii), so we set $v_i:=r_i$.  Finally, since $v_1\in V(T_1)$ and $v_2\in V(T_2)$, $v_1$ and $v_2$ are unrelated.
\end{proof}

\subsection{A Connectivity Lemma}

The \defin{$x\times y$ grid} $G_{x\times y}$ is the graph with vertex-set $V(G_{x\times y}):=\{1,\ldots,x\}\times\{1,\ldots,y\}$ and that contains an edge with endpoints $(x_1,y_1)$ and $(x_2,y_2)$ if and only if $|x_1-x_2|+|y_1-y_2|=1$.  An edge of $G_{x\times y}$ is \defin{horizontal} if its two endpoints agree in the second (y) coordinate.  For each $i\in\{1,\ldots,x\}$, the vertex-set $\{i\}\times\{1,\ldots,y\}$ is called \defin{column $i$} of $G_{x\times y}$.  A set $C$ of columns is \defin{consecutive} if $G_{x\times y}[\cup C]$ is connected.

\begin{lem}\label{grid_connectivity}
  Let $x,y,p\ge 1$ be integers, let $G$ be a graph obtained by subdividing horizontal edges of $G_{x\times y}$, and let $S\subseteq V(G)\setminus V(G_{x\times y})$ be a set of subdivision vertices of size $|S|< py$.  Then some component of $G-S$ contains at least $x/p$ consecutive columns of $G_{x\times y}$.
\end{lem}

\begin{proof}
  For each $i\in\{1,\ldots,x-1\}$, in order to separate column $i$ from column $i+1$, $S$ must contain at least $y$ subdivision vertices on the horizontal edges between columns $i$ and $i+1$.  Since $|S|< py$, this implies that there are at most $p-1$ values of $i\in\{1,\ldots,x-1\}$ for which columns $i$ and $i+1$ are in different components of $G-S$. These at most $p-1$ values of $i$ partition $\{1,\ldots,x\}$ into at most $p$ intervals, at least one of which contains at least $x/p$ consecutive columns that are contained in a single component of $G-S$.
\end{proof}

\subsection{The Proof}

Recall that, for each $h\in\N$, $G_{h}$ is the planar supergraph of the complete binary tree $T_h$ of height $h$ obtained by adding the edges of a path $D_i$ that contains all vertices of depth $i$, in left-to-right order, for each $i\in\{1,\ldots,h\}$.   Since $T_h$ is a spanning subgraph of $G_h$, the \defin{depth} of a vertex $v$ in $G_h$ refers to the depth of $v$ in $T_h$.  The \defin{height} of a depth-$d$ vertex of $T_h$ is $h-d$.
We are now ready to prove the following result that, combined with \cref{ding_translation} is sufficient to prove \cref{main_thm}:

\begin{thm}\label{main_thm_tree}
  For every $h\in\N$, every tree $T$, and every path $P$, if $G_{h}\subsetcong T\boxtimes P\boxtimes K_c$ then $c\ge 2^{\Omega(\sqrt{\log h})}$.
\end{thm}

It is worth noting that, unlike \cref{main_thm}, there is no restriction on the maximum degree of the tree $T$.

Before diving into technical details, we first sketch our strategy for proving \cref{main_thm_tree}.  We may assume that $c \le 2^{\sqrt{\log_2 h}}$ since, otherwise there is nothing to prove.  Recall \cref{partitions_vs_products}, which states that if $G_h\subsetcong T\boxtimes P\boxtimes K_c$ then $G_h$ has a tree-partition $\mathcal{T}:=\{B_x:x\in V(T)\}$ and a path-partition (that is, layering) $\mathcal{P}:=\{P_y:y\in V(P)\}$ with $|B_x\cap P_y|\le c$ for each $(x,y)\in V(T)\times V(P)$. However, since $G_h$ has diameter $2h$, \cref{diameter_spread} (with $R=B_x$) implies that $|B_x|\le c(2h+1)$ for each $x\in V(T)$. This upper bound on $|B_x|$ is used to establish all of the results described in the following paragraph.

\begin{figure}
  \begin{center}
    \includegraphics[width=.95\textwidth]{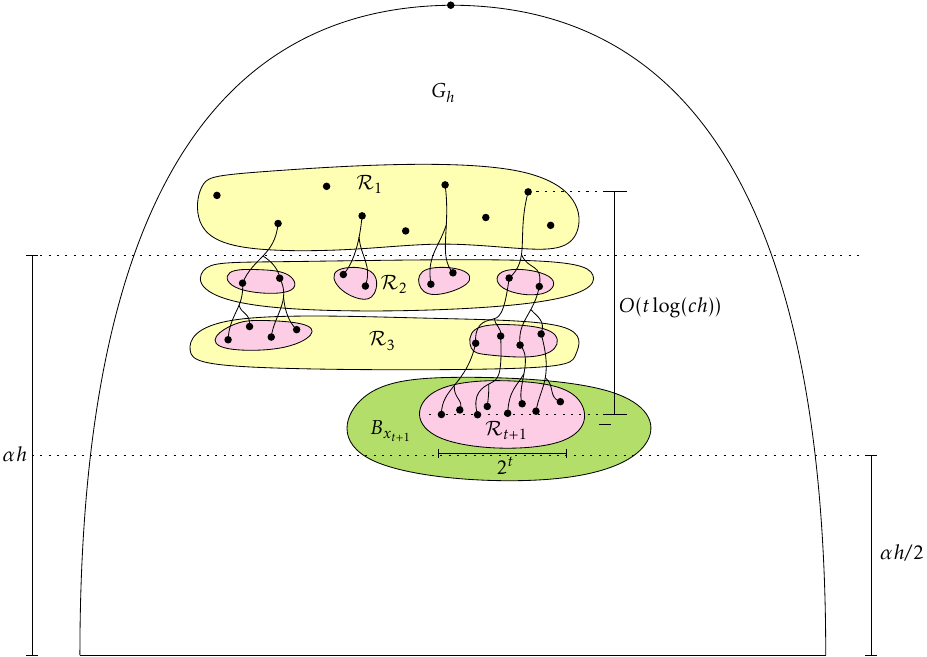}
  \end{center}
  \caption{The proof of \cref{main_thm_tree}.}
  \label{sketch}
\end{figure}

Refer to \cref{sketch}.  We will construct a sequence of sets $\mathcal{R}_1,\ldots,\mathcal{R}_{t+1}$ and a sequence of nodes $x_1,\ldots,x_{t+1}$ of $T$, where each $\mathcal{R}_i$ is a family of unrelated sets in $T_h$ such that $\cup\mathcal{R}_i\subseteq B_{x_i}$. The first family $\mathcal{R}_1$ consists of $q_1\ge h/(25c)$ singleton sets whose union is an unrelated set in $T_h$. For each $i\in\{2,\ldots,t+1\}$, $\mathcal{R}_i$ has size $q_i\ge q_1/(10c)^{i-1}-3$. For each $\mathcal{R}_i:=\{R_{i,1},\ldots,R_{i,q_i}\}$, each $R_{i,j}\subseteq V(T)$ is an unrelated set in $T_h$ of size $2^{i-1}$ that has a common ancestor $a_{i,j}$ of height at least $h/5$ that is at distance at most $(i-1)(\log_2(ch)+2)$ from every element in $R_{i,j}$. Furthermore, $\{a_{i,1},\ldots,a_{i,q_{i}}\}$ is an unrelated set. These properties imply that $\cup\mathcal{R}_i$ is also an unrelated set.

We do this for some appropriately chosen integer $t\in\Theta(\sqrt{\log h})$ in order to ensure that $q_{t+1}\ge 1$, so $\mathcal{R}_{t+1}$ contains at least one part $R$ of size $2^t$.
By \cref{diameter_spread}, there exists some $y \in V(P)$ such that
\[
  |R\cap P_y| \ge \frac{|R|}{\diam_{G_h}(R)+1} \ge \frac{2^t}{2t(\log_2(ch)+2)+1}
  = 2^{t-\log_2(t\log_2(ch))-O(1)} =  2^{\Omega(t)} = 2^{\Omega(\sqrt{\log_2 h})} \enspace .
\]
Since $R\subseteq B_{x_{t+1}}$, $|B_{x_{t+1}}\cap P_y|\ge 2^{\Omega(\sqrt{\log h})}$. Since $c\ge \max\{|B_x\cap P_y|:(x,y)\in V(T)\times V(P)\}$, the assumption that $c\le 2^{\sqrt{\log h}}$ therefore leads to the conclusion that $c\ge 2^{\Omega(\sqrt{\log h})}$, which establishes \cref{main_thm}.

We now proceed with the details of the proof outlined above.  The next two lemmas will be used to obtain the set $\mathcal{R}_1$ that allows us to start the argument.  Informally, the first lemma says that every balanced separator $S$ of $G_h$ must contain a vertex of depth $i$ for each $i\in\{i_0,\ldots,h\}$, where $i_0\in O(\log|S|)$.

\begin{lem}\label{small_depth_separator}
  Let $h\in\N$ with $h\geq 1$, let $S\subseteq V(G_h)$, $S\neq\emptyset$.  If $G_h-S$ has no component with more than $|V(G_h)|/2$ vertices then $S\cap V(D_i)\neq\emptyset$ for each $i\in\{i_0,\ldots,h\}$, where $i_0:=\ceil{\max\{2\log_2|S|+2, \log_2(1+(h+2)|S|)-1\}}$.
\end{lem}

\begin{proof}
  Let $C$ be the vertex set of a component of $G_h-S$ that maximizes $C\cap V(D_h)$. For each $i\in\{0,\ldots,h\}$, let $C_i:=C\cap V(D_i)$ and let $S_i:=S\cap V(D_i)$. We will show that, for each $i\ge i_0$, $C_i$ is non-empty but does not contain all $2^i$ vertices in $D_i$.  Therefore $S_i\supseteq N_{G_h}(C_i)\cap V(D_i)\neq\emptyset$ for each $i\in\{i_0,\ldots,h\}$.

  For each $i\in\{0,\ldots,h-1\}$, the vertices in $C_{i+1}$ are adjacent to at least $|C_{i+1}|/2$ vertices of $D_{i}$, so $|C_i|\ge |N_{G_h}(C_{i+1})\cap V(D_i)\setminus S_i| \ge |C_{i+1}|/2 - |S_i|$.  Iterating this inequality $h-i$ times gives $|C_i|\ge |C_h|/2^{h-i}-\sum_{j=i}^{h-1}|S_j|/2^{h-i-1}\ge |C_h|/2^{h-i}-|S|$.  The vertices in $S$ partition $V(D_h)\setminus S$ into at most $|S|+1$ connected components.  Since $C$ is chosen to maximize $|C_h|$, $|C_h| \ge (2^h-|S|)/(|S|+1) > 2^{h}/(|S|+1) - 1$.  Therefore,
  \begin{equation}
    |C_i|\ge \frac{|C_h|}{2^{h-i}} - |S| > \frac{2^h/(|S|+1)-1}{2^{h-i}} - |S|
    \ge 2^{i-\log_2(|S|+1)}-|S|-1 \ge 0 \label{first}
  \end{equation}
  for $i\ge 2\log_2 |S|+2$.  Since $i_0\ge 2\log_2 |S|+2$, this establishes that $C_i$ is non-empty for each $i\in\{i_0,\ldots,h\}$.

  For each $i\in\{0,\ldots,h-1\}$, the vertices in $C_i$ are adjacent to at least $2|C_i|$ vertices of $D_{i+1}$, so $|C_{i+1}|\ge 2|C_i|-|S_{i+1}|$.  Iterating this $h-i$ times gives:
  \begin{equation}
    |C_h| \ge 2^{h-i}|C_i| - \sum_{j=i+1}^h 2^{h-j}|S_j| \ge 2^{h-i}|C_i| - 2^{h-i-1}|S| \enspace . \label{ch_lower_bound}
  \end{equation}
  Suppose that $|C_{i^*}|=2^{i^*}$ for some $i^*\in\{0,\ldots,h\}$. Then \cref{ch_lower_bound} implies that $|C_h|\ge 2^h-2^{h-i^*-1}|S|$.  Therefore, by \cref{first},
  \begin{equation}
      |C|
         =\sum_{i=0}^h |C_i| \ge \sum_{i=0}^h\left(\frac{|C_h|}{2^{h-i}} - |S|\right)
       > 2|C_h| - 1 -(h+1)|S|
       \ge 2^{h+1}-2^{h-i^*}|S|-1 - (h+1)|S| \enspace .  \label{crux}
  \end{equation}
  However, $2^h > |V(G_h)|/2 \ge |C|$, and combining this with \cref{crux} gives $2^h > 2^{h+1} - 2^{h-i^*}|S|-1-(h+1)|S|$.
  Rewriting this inequality, we get
  \begin{equation}
     2^h < 2^{h-i^*}|S|+ 1 + (h+1)|S| \enspace . \label{cruxi}
  \end{equation}
  Multiplying each side of \cref{cruxi} by $2^{i^*-h}$ then gives:
  \begin{align*}
    2^{i^*} & <  |S| + 2^{i^*-h}(1+(h+1)|S|) \\
            & \le |S| + 1 + (h+1)|S| & \text{(since $i^*\le h$, so $2^{i^*-h}\le 1$)} \\
            & = 1+ (h+2)|S| \enspace .
  \end{align*}
  Taking the logarithm of each side then gives $i^* < \log_2(1+(h+2)|S|)\le i_0$.  This establishes that $|C_i|< 2^i$ for each $i\in\{i_0,\ldots,h\}$ and  completes the proof.
\end{proof}

The following lemma shows that every tree-partition of $G_h$ must have a part with a large unrelated set that is far from the leaves of $T_h$ and will be used to obtain our first set $\mathcal{R}_1$.

\begin{lem}\label{startup}
  For every $\alpha\in(0,1/4)$, there exists $h_0$ such that the following is true, for all integers $h\ge  h_0$ and all $c\in[1,h]$.  If $\mathcal{T}:=\{B_x:x\in V(T)\}$ is a tree-partition of $G_h$ of width less than $ch$ then there exists a node $x\in V(T)$ and a subset $R\subseteq B_x$ such that
  \begin{compactenum}[(i)]
    \item $R$ is unrelated;
    \item $|R|\ge \alpha^2 h/c$; and
    \item Each vertex in $R$ has height at least $\alpha h$.
  \end{compactenum}
\end{lem}

\begin{proof}
  It is well-known and easy to show that there exists a node $x$ of $T$ such that $G-B_x$ has no component with more than $|V(G_h)|/2$ vertices \cite[(2.6)]{robertson.seymour:graph}. Let $Y$ be the set of vertices in $B_x$ that have height at least $h/4$.  By \cref{small_depth_separator}, $|Y|\ge 3h/4 - O(\log (ch+1))$.

  Let $T_Y$ be the minimal (connected) subtree of $T_h$ that spans $Y$, and let $L$ be the set of leaves of $T_Y$ (excluding the root of $T_Y$ if this happens to be contained in $Y$).  Observe that $L\subseteq Y$ is an unrelated set. Therefore, $L$ satisfies (i) and, by definition, each vertex in $L$ has height at least $h/4 > \alpha h$, so $L$ satisfies (iii).  If $|L|\ge \alpha h \ge \alpha^2 h/c$ then $L$ also satisfies (ii).  In this case, we can take $R:=L$ and we are done.  We now assume that $|L|< \alpha h$.

  Let $Z$ consist of all vertices in $V(T_h)\setminus V(T_Y)$ whose parents are in $Y\setminus L$.   Observe that $Z$ is an unrelated set of vertices each having height at least $h/4$. For each $v$ of $T_Y$, let $d_v$ denote the number of children of $v$ in $T_Y$.  Then,
  \[
     \sum_{v\in Y\setminus L} (d_v-1)
     \le \sum_{v\in V(T_Y)\setminus L} (d_v-1)
     = |L|-1 \enspace ,
  \]
  where the second equality is a standard fact about rooted trees.
  Rewriting this, we get $\sum_{v\in Y\setminus L} {d_v} < |Y\setminus L| + |L| = |Y|$.  On the other hand, each $v\in Y\setminus L$ contributes $2-d_v$ vertices to $Z$, so
  \[
    |Z| = \sum_{v\in Y\setminus L} (2-d_v) \enspace .
  \]
  Combining these two formulas, we obtain
  \[
    |Z| \ge 2|Y\setminus L| - |Y| = |Y| - 2|L|
    \ge 3h/4-O(\log(ch+1)) - 2\alpha h
    \ge h/4-O(\log(ch+1)) \enspace .
  \]
  Refer to \cref{paths}.  For each $r\in Z$, \cref{one_path} applied to the subtree of $T_h$ rooted at $r$ with $S=B_x$ implies that $r$ has a descendant $v$ such that
  \begin{inparaenum}[(a)]
    \item the parent of $v$ is in $B_x\cup\{r\}$;
    \item the height of $v$ is at least $h/4-O(\log(ch+1))$; and
    \item $T_h-B_x$ contains a path $Q_{v}$ from $v$ to a leaf of $T_h$.
  \end{inparaenum}

  Form the set $Z'$ using the following rule for each $r\in Z$: If the vertex $v$ described in the preceding paragraph is a child of $r$ then place $r$ in $Z'$, otherwise place $v$ in $Z'$. Since each $r\in Z$ is a child of some vertex in $Y\subseteq B_x$, this ensures that the parent of $v$ is in $B_x$ for each $v\in Z'$.  Since $Z$ is an unrelated set and $Z'$ is obtained by replacing each vertex in $Z$ with one of its descendants, $Z'$ is an unrelated set.  Since $\alpha < 1/4$, for sufficiently large $h$, $|Z'|\ge h/4 - O(\log(ch+1)) \ge \alpha h$ and each vertex in $Z'$ has height at least $h/4 - O(\log(ch+1)) \ge \alpha h$.

  \begin{figure}
    \begin{center}
      \includegraphics{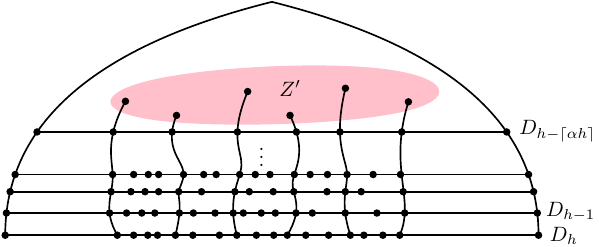}
    \end{center}
    \caption{A step in the proof of \cref{startup}.}
    \label{paths}
  \end{figure}

  Now observe that the union of the paths $Q_{v}$ for each $v\in Z'$ and the paths $D_{h-\lceil\alpha h\rceil+1},\ldots,D_{h}$ contains a subgraph $G'$ isomorphic to a graph that can be obtained from the grid $G_{\lceil \alpha h\rceil\times\lceil \alpha h\rceil}$ by subdividing horizontal edges.  Since $B_x$ does not contain any vertex of $Q_{v}$ for each $v\in Z'$,  $B_x\cap V(G')$ contains only vertices corresponding to subdivision vertices.  Therefore, by \cref{grid_connectivity}, some component of $G-B_x$ contains a subset $R\subseteq Z'$ of size at least $\alpha^2 h/c$.  Each element in $R$ has a parent in $B_x$.  By \cref{tree_thingy} some neighbour $y$ of $x$ in $T$ has a bag $B_y$ that contains all of $R$.  This completes the proof.
\end{proof}

A set $\mathcal{R}:=\{R_1,\ldots,R_q\}$ of subsets of $V(T_h)$ is \defin{$(k,\ell,m)$-compact} if it has the following properties:

\begin{compactenum}
  \item For each $i\in\{1,\ldots,q\}$, $R_i$ is unrelated and $|R_i|\ge k$.
  \item For each $i\in\{1,\ldots,q\}$ there exists a common ancestor $a_i$ of $R_i$ such that $\dist_{T_h}(v,a_i)\le\ell$ for each $v\in R_i$.
  \item $a_1,\ldots,a_q$ are unrelated and each has height at least $m$.
\end{compactenum}

This definition has the following implications:
\begin{inparaenum}[(i)]
  \item $\cup \mathcal{R}$ is an unrelated set; and
  \item If $a_i$ precedes $a_j$ in the left-to-right ordering of $\{a_1,\ldots,a_q\}$ then every element of $R_i$ precedes every element of $R_j$ in the left-to-right order of $\cup\mathcal{R}$.
\end{inparaenum}
We say that a vertex $v$ of $T_h$ is \defin{compatible} with $S\subseteq V(T_h)$ if the parent of $v$ is in $S$ and $T_h-S$ contains a path from $v$ to a leaf of $T_h$.  A $(k,\ell,m)$-compact set $\mathcal{R}$ is \defin{compatible} with $S$ if each vertex in $\cup\mathcal{R}$ is compatible with $S$.

\begin{lem}\label{compatible_set}
  Let $\mathcal{R}:=\{R_1,\ldots,R_q\}$ be a $(k,\ell,m)$-compact set,  and let $S\supseteq \cup\mathcal{R}$ have size $1\le |S|< 2^{m-\ell-2}$.  Then, there exists a $(2k,\ell + \log_2|S|+2,m)$-compact set $\mathcal{R}':=\{R_1',\ldots,R_q'\}$ that is compatible with $S$.
\end{lem}

\begin{proof}
  For each $i\in\{1,\ldots,q\}$ and each $r\in R_i$, replace $r$ with the descendants $v_1$ and $v_2$ of $r$ described in \cref{two_paths} and call the resulting set $R_i'$.   Then $|R_i'|=2|R_i|\ge 2k$ and $\dist_{T_h}(v,a_i)\le \ell+\log_2|S|+2$ for each $v\in R_i'$, where $a_i$ is the common ancestor of $R_i$ in the definition of $(k,\ell,m)$-compact.  Therefore $\mathcal{R}':=\{R_1',\ldots,R_q'\}$ is a $(2k,\ell + \log_2(|R|)+2,m)$-compact set.
\end{proof}

The next lemma is the last ingredient in the proof of \cref{main_thm}.
\begin{lem}\label{big_lemma}
  Let $\mathcal{T}:=\{B_x:x\in V(T)\}$ be a tree-partition of $G_h$ and let $\mathcal{P}:=\{P_y:y\in V(P)\}$ be a path-partition of $G_h$.  Then there exists $(x,y)\in V(T)\times V(P)$ such that $|B_x\cap P_y| \ge 2^{\Omega(\sqrt{\log h})}$.
\end{lem}

\begin{proof}
  Let $x\in V(T)$ be a node that maximizes $|B_x|$.  Then $\diam_{G_h}(B_x)\le\diam_{T_h}(B_x) \le 2h$ so, by \cref{diameter_spread}, $|B_x\cap P_y|\ge |B_x|/(2h+1)$ for some $y\in V(P)$.  If $|B_x|\ge h2^{\sqrt{\log_2 h}}$ then there is nothing more to prove, so we may assume that $|B_x| < ch$ where $c:= 2^{\sqrt{\log_2 h}}$.  Note that $c\ge 1$ for every $h\ge 1$.

  By \cref{startup}, with $\alpha:= 1/5$, $T$ contains a node $x_1$ such that $B_{x_1}$ contains an unrelated set $R$ of size $q_1:=|R|\ge h/(25c)$ where each vertex in $R$ has height at least $m:=m_1:=h/5$.  Let  $\mathcal{R}_1:=\{\{v\}:v\in R\}$.  By definition $\mathcal{R}_1$ is a $(1,0,m)$-compact set.  $\mathcal{R}_1$ will be the first in a sequence of sets $\mathcal{R}_1,\ldots,\mathcal{R}_{t+1}$, where $t$ will be fixed below. For each $i\in\{1,\ldots,t+1\}$, $\mathcal{R}_i$ will satisfy the following properties:
  \begin{enumerate}[(a)]
     \item $\mathcal{R}_i$ is a $(2^{i-1},(i-1)(\log_2(ch)+2),m_i)$-compact set, where $m_i \ge h/5-(i-1)(\floor{\log_2(ch)}+2)$. \label{ri_compact}
     \item $q_i:=|\mathcal{R}_i|$, with $q_i > q_{i-1}/(10c) - 2$ if $i\ge 2$. \label{ri_size}
     \item There exists $x_i\in V(T)$ such that $\cup\mathcal{R}_i\subseteq B_{x_i}$. \label{ri_containment}
  \end{enumerate}
  Note that, by a simple inductive argument, one can show that
  \[
  q_i > q_1/(10c)^{i-1} - 3 \enspace .
  \]
  Indeed, the base case $i=1$ holds trivially, and for the inductive case ($i\geq 2$) we have $q_i > q_{i-1}/(10c) - 2 > (q_1/(10c)^{i-2} - 3)/(10c) -2 > q_1/(10c)^{i-1} - 3$.

  It is straightforward to verify that $\mathcal{R}_1$ satisfies (\ref{ri_compact})--(\ref{ri_containment}).
  Let $t:= \min\{t_1,t_2\}$ where $t_1:=\floor{\log_{10c}(q_1/3)}$ and $t_2:=\floor{h/(10(\sqrt{\log_2 h} + \log_2 h) + 2)}$. Observe that, since $c=2^{\sqrt{\log_2 h}}$,  $t_1\ge \log_{10c}(h/75c)\in\Omega(\log_c h)\subseteq \Omega(\sqrt{\log h})$ and that $t_2\ge h/(10(\sqrt{\log_2 h} + \log_2 h) + 2)-1\in\Omega(h/\log h)$.  Therefore $t\in\Omega(\sqrt{\log h})$. These specific values of $t_1$ and $t_2$ are chosen for the following reasons:
  \begin{compactenum}[(i)]
    \item Since $t\le t_1$, $q_{t+1} > q_1/(10c)^{t_1} - 3  \ge 0$, so $q_{t+1}\ge 1$.
    \item Since $t\le t_2$, $m_i\ge h/5-t_2(\floor{\log_2(ch)}+2) \ge h/10$ for each $i\in\{2,\ldots,t+1\}$.
  \end{compactenum}
  We now describe how to obtain $\mathcal{R}_{i+1}$ from $\mathcal{R}_i$ for each $i\in\{1,\ldots,t\}$.  By \cref{compatible_set} (applied to $\mathcal{R}:=\mathcal{R}_i$ and $S:=B_{x_i}$), $T_h$ contains a $(2^i,i\log_2(ch)+2,m)$-compact set $\mathcal{R}_{i+1}^+$ of size $q_i$ that is compatible with $B_{x_i}$. For each $v\in\cup\mathcal{R}_{i+1}^+$, $v$ has height at least $m_{i+1}:=m_i-(\floor{\log_2(ch)}+2)\ge h/5-i(\floor{\log_2(ch)}+2)$.  Therefore $\mathcal{R}_{i+1}^+$ satisfies (\ref{ri_compact}), %
  but does not necessarily satisfy (\ref{ri_containment}).  Next we show how to extract $\mathcal{R}_{i+1}\subseteq\mathcal{R}_{i+1}^+$ that also satisfies (\ref{ri_size}) and (\ref{ri_containment}).

  For each $v\in \cup\mathcal{R}_{i+1}^+$, $T_h-B_{x_i}$ contains a path $Q_v$ from $v$ to a leaf of $T_h$.  The union of the paths in $D_{h-m_{i+1}},\ldots,D_{h}$ and the paths in $\mathcal{C}_i:=\{Q_v:v\in\cup\mathcal{R}_{i+1}^+\}$ contains a subgraph $G'$ isomorphic to a graph that can be obtained from $G_{2^{i}q_i\times m_{i+1}}$ by subdividing horizontal edges.  By \cref{grid_connectivity} applied to $G:=G'$ with $S:=B_{x_i}$ and $p:=ch/m_{i+1}$, some component $X'$ of $G'-B_{x_i}$ contains  $q_i'\ge 2^{i}q_im_{i+1}/(ch)\ge 2^iq_i/(10c)$ consecutive columns $C_1,\ldots,C_{q_i'}$ of $G'$.  The component $X'$ is contained in some component $X$ of $G_h-B_{x_i}$.

  Since $\cup\mathcal{R}_i$ is unrelated, it has a left to-right-order. This order defines a total order $\prec$ on the paths in $\mathcal{C}_i$, in which $Q_v\prec Q_w$ if and only if $v$ precedes $w$ in left-to-right order. The resulting total order $(\prec,\mathcal{C}_i)$ corresponds to the order of the columns in $G'$ and each part in $\mathcal{R}_{i+1}^+$ corresponds to $2^i$ consecutive columns of $G'$. There are at most two parts $R\in\mathcal{R}_i$ such that $0 < |R\cap (C_1\cup\cdots\cup C_{q_i'})| < |R|$.  These two parts account for at most $2(2^{i}-1)$ of the columns in $C_1,\ldots, C_{q_i'}$.  Therefore, the number of parts of $\mathcal{R}_{i+1}^+$ completely contained in $C_1\cup\cdots\cup C_{q_i'}$ is at least
  \begin{align*}
      (q_i'-(2^{i+1}-2))/2^i & > q_i/(10c) - 2 \enspace .
  \end{align*}
  We define $\mathcal{R}_{i+1}\subseteq\mathcal{R}_{i+1}^+$ as the set of parts in $\mathcal{R}_{i+1}^+$ that are completely contained in $C_1\cup\cdots\cup C_{q_i'}$.  The preceding calculation shows that $\mathcal{R}_{i+1}$ satisfies (\ref{ri_size}).  Since $\cup\mathcal{R}_{i+1}$ is contained in a single component $X$ of $G_x-B_{x_i}$ and each vertex in $\cup\mathcal{R}_{i+1}$ has a neighbour (its parent in $T$) in $B_{x_i}$, \cref{tree_thingy} implies that $T$ contains an edge $x_{i}x_{i+1}$ with $\cup\mathcal{R}_{i+1}\subseteq B_{x_{i+1}}$. Therefore $\mathcal{R}_{i+1}$ satisfies (\ref{ri_containment}).

  This completes the definition of $\mathcal{R}_1,\ldots,\mathcal{R}_{t+1}$. Properties~(\ref{ri_compact})--(\ref{ri_containment}) imply that, $\mathcal{R}_{t+1}$ is a $(2^t,t(\log_2(ch)+2),m)$-compact set of size $q_{t+1} \ge 1$ and $\cup\mathcal{R}_{t+1}\subseteq B_{x_{t+1}}$ for some $x_{t+1}\in V(T)$.
  Let $R$ be one of the sets in $\mathcal{R}_{t+1}$.  Since $\mathcal{R}_{t+1}$ is $(2^t,t(\log_2(ch)+2),m)$-compact, $|R|\ge 2^t$ and all vertices in $R$ have a common ancestor $a$ whose distance to each element of $R$ is at most $t(\log_2(ch)+2)$.  Therefore, $\diam_{G_h}(R)\le\diam_{T_h}(R)\le 2t(\log_2(ch)+2)$. By \cref{diameter_spread}, there exists some $P_y\in\mathcal{P}$ with
  \begin{align*}
    |B_{x_{t+1}}\cap P_y|
    & \ge |R\cap P_y| \\
    & \ge \frac{|R|}{\diam_{G_h}(R)+1} \\
    & \ge \frac{2^t}{2t(\log_2(ch)+2)+1} \\
    & = 2^{t-O(\log\log h)} \\
    & = 2^{\Omega(\sqrt{\log h})-O(\log\log h)}
   = 2^{\Omega(\sqrt{\log h})} \enspace . \qedhere
  \end{align*}
\end{proof}

\begin{proof}[Proof of \cref{main_thm_tree}]
  Suppose $G_h\subsetcong T\boxtimes P\boxtimes K_c$ for some tree $T$ and some path $P$.  By \cref{partitions_vs_products}, $G_h$ has a $T$-partition $\mathcal{T}:=\{B_x:x\in V(T)\}$ and a path-partition $\mathcal{P}:=\{P_y:y\in V(P)\}$ such that $|B_x\cap P_y|\le c$ for each $(x,y)\in V(T)\times V(P)$.  By \cref{big_lemma}, there exists $(x,y)\in V(T)\times V(P)$, such that $|B_x\cap P_y| \ge 2^{\Omega(\sqrt{\log h})}$.  Combining these upper and lower bounds on $|B_x\cap P_y|$ implies that $c\ge 2^{\Omega(\sqrt{\log h})}$.
\end{proof}

\begin{proof}[Proof of \cref{main_thm}]
  Let $n:=|V(G_h)|=2^{h+1}-1$.  Suppose that $G_h\subsetcong H\boxtimes P\boxtimes K_c$ for some graph $H$ of treewidth $t$ and maximum degree $\Delta$, some path $P$ and some integer $c$.  Then, by \cref{dingy,partitions_vs_products}, $G\subsetcong T\boxtimes P\boxtimes K_{24 c \Delta (t+1)}$ for some tree $T$.  By \cref{main_thm_tree} $24 c \Delta (t+1) \in \Omega(2^{\sqrt{\log h}})$, so $c\Delta t \ge 2^{\Omega(\sqrt{\log h})} = 2^{\Omega(\sqrt{\log\log n})}$.
\end{proof}

\section{Open Problems}

We know that every planar graph $G$ is contained in a product of the form $H\boxtimes P\boxtimes K_3$ where $\tw(H)\le 3$ \cite{dujmovic.joret.ea:planar}. \cref{main_thm_tree} states that, for every $c$, there exists a planar graph of maximum degree $5$ that is not contained in any product of the form $T\boxtimes P\boxtimes K_c$ where $T$ is a tree and $P$ is a path.  This leaves the following open problem:

\begin{quote}

  Is every planar graph $G$ of maximum degree $\Delta$ contained in a product of the form $H\boxtimes P\boxtimes K_c$ where the treewidth of $H$ is $2$, $P$ is a path, and $c$ is some function of $\Delta$?
\end{quote}

\cref{two_tree} and \cref{partitions_vs_products} show that $G_h$ is a subgraph of $H\boxtimes P$ where $H$ has treewidth $2$ (and is even outerplanar) and $P$ is a path. Our proof breaks down in this case because, unlike tree-partitions, outerplanar-partitions do not satisfy \cref{tree_thingy}.  Indeed, the outerplanar-partition illustrated in \cref{two_tree} contains a part $B_x$ and a component $X$ of $G_h-B_x$ with $|N_{G_h}(B_x)\cap V(X)|=h$.  In a tree-partition this would imply that $|N_{G_h}(B_x)\cap V(X)\cap B_y|=h$ for some other part $B_y$ of the partition.  In contrast, for the outerplanar-partition shown in \cref{two_tree}, $|N_{G_h}(B_x)\cap V(X)\cap B_y|\le 1$ for each $y\in V(H)$.

\begin{figure}
  \begin{center}
    \begin{tabular}{cc}
      \includegraphics[scale=.7]{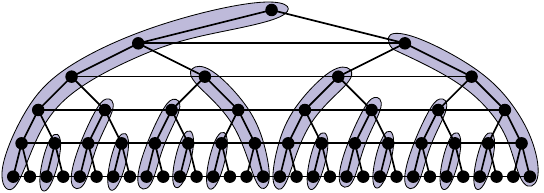} & \includegraphics[scale=.7]{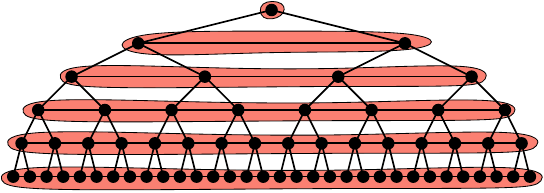}
    \end{tabular}
  \end{center}
  \caption{An outerplanar-partition $\mathcal{H}$ and a path-partition $\mathcal{P}$ of $G_5$ for which $|B\cap P|\le 1$ for each $B\in\mathcal{H}$ and $P\in\mathcal{P}$.}
  \label{two_tree}
\end{figure}

\section*{Acknowledgements}

Much of this research was done during the First Adriatic Workshop on Graphs and Probability, June 5--10, in Hvar, Croatia.  The authors are grateful to the organizers and other workshop participants for providing an optimal working environment.  The authors would especially like to thank Nina~Kamčev for helpful discussions.

\bibliographystyle{plainurlnat}
\bibliography{tp}

\end{document}